\documentclass[11pt]{amsart}
\usepackage{amssymb}
\usepackage{mathrsfs}
\usepackage{enumerate}
\makeatletter
\@namedef{subjclassname@2010}{%
  \textup{2010} Mathematics Subject Classification}
\makeatother
\newtheorem{lemma}{Lemma}[section]
\newtheorem{theorem}{Theorem}[section]
\newtheorem{proposition}{Proposition}[section]
\newtheorem{corollary}{Corollary}[section]
\newtheorem{definition}{Definition}[section]
\newtheorem{remark}{Remark}[section]

\begin{document}

\title{Equal-norm Parseval $K$-frames  in Hilbert spaces}

\author[V. Sadri]{Vahid Sadri$^*$}
\address{Department of Mathematics, Faculty of Tabriz  Branch\\ Technical and Vocational University (TVU), East Azarbaijan, Iran.}
\email{vahidsadri57@gmail.com}

\author[Gh. Rahimlou]{Gholamreza Rahimlou}%
\address{Department of Mathematics, Faculty of Tabriz  Branch\\ Technical and Vocational University (TVU), East Azarbaijan, Iran.}
\email{grahimlou@gmail.com}

\begin{abstract}
In this paper, we focus on  frames of operators or $K$-frames on Hilbert spaces in  Parseval cases. Since equal-norm tight frames play important roles for robust data transmission, we aim to study this topics on Parseval $K$-frames. We will show that each finite set of equal-norms of vectors can be extended to an equal-norm $K$-frame. Also, we present a correspondence between Parseval $K$-frames and the set of all closed subspaces of a finite Hilbert space. Finally, a construction of equal-norm dual $K$-frames will be introduced.
\end{abstract}


\keywords{Perseval frames, $K$-frames, $K$-dual frames.\\
*Corresponding Author.}

\maketitle

\section{Introduction}
The concept of frames in Hilbert spaces has been introduced by Duffin and Schaeffer \cite{ds} in 1952 to study some deep problems in nonharmonic Fourier series.
Frames have a significant role in both pure and applied mathematics, so that these are a fundamental research area in mathematics, computer science and quantum information and several new applications have been developed, e. g. besides traditional application as signal processing, image processing, data compression and sampling theory.

Frames are redundant sets of vectors in a Hilbert space which yield one natural representation for each vector in the space
but which may have infinitely many different representations for a given  vector \cite{bolc1,bolc2,casazza1,casazza3,cevet,dau,heil,stero,ds,hol}.
Recently, several new applications for (uniform tight) frames have been developed.
The first developed by Goyal, Kova\^cevi\'c, and Vetterli \cite{kova1,kova2,kova3}, uses the redundancy of a frame to mitigate the effect of losses in packet-based communication systems. Modern communication networks transport packets of data from a source to a recipient. These packets are sequences of information bits of a certain length surrounded by error-control, addressing, and timing information that assure that the packet is delivered without errors. This is accomplished by not delivering the packet if it contains errors. Failures here are due primarily to buffer overflows at intermediate nodes in the network. So to most users, the behavior of a packet network is not characterized by random loss, but by unpredictable transport time This is due to a protocol, invisible to the user, that retransmits lost packets. Retransmission of packets takes much longer than the original transmission in many applications, retransmission of lost packets is not feasible and the potential for large delay is unacceptable.
Another recent important application of uniform normalized tight frames is in multiple-antenna code design \cite{hassib}.
Much theoretical work has been done to show that communication systems which employ multiple antennas can have very high
channel capacities \cite{fos,telat}.

 Frames for operators or $K$-frames have been introduced by G\u avru\c ta in \cite{gal} to study the nature of atomic systems for a separable Hilbert space  with respect to a bounded linear operator $K$. It is a well-known fact that $K$-frames are more general than the classical frames and due to higher generality of $K$-frames, many properties of frames may not hold for $K$-frames.

In this paper, we will prove some fundamental results for the parseval $K$-frames, which have already been proved for tight frames in \cite{casazza2} and we will show that each finite set of $K$-norm vectors can be extended to become a $K$-norm of a $K$-frame.

Throughout this paper, $H$, $H_1,$ and $H_2$ are separable Hilbert spaces and $\mathcal{B}(H_1,H_2)$ is the collection of all the bounded linear operators of $H_1$ into $H_2$. If $H_1=H_2=H$, then $\mathcal{B}(H,H)$ will be denoted by $\mathcal{B}(H)$.  The notation $W\prec H$ means that $W$ is a closed subspace of $H$, also $P_{W}$ is the orthogonal projection from $H$ onto a closed subspace $W\prec H$ and  $\lbrace H_j\rbrace_{j\in\Bbb J}$ is a sequence of Hilbert spaces where $\Bbb J$ is a subset of $\Bbb Z$. When $\dim H<\infty$, we say that $H$ is the finite Hilbert space.

\section{Review of operator theory and $K$-frames}
In this part, we aim to review of some topics and notations about bounded operators and  frames on Hilbert spaces.

Let $U$ be an operator with closed range (or, $\mathcal{R}(U)$ be closed), then there exists a
\textit{right-inverse operator} $U^ \dagger$ (pseudo-inverse of $U$) in the
following senses.
\begin{lemma}\cite{ch}\label{lem1}
Let $U\in\mathcal{B}(H_1,H_2)$  be a bounded operator with closed range $\mathcal{R}(U)$. Then there exists a bounded operator $U^\dagger \in\mathcal{B}(H_2,H_1)$ for which
$$UU^{\dagger} x=x, \ \ x\in \mathcal{R}(U).$$
\end{lemma}
\begin{lemma}\cite{ch}\label{lem2}
Let $U\in\mathcal{B}(H_1,H_2)$. Then the following assertions hold:
\begin{enumerate} \item
 $\mathcal{R}(U)$ is closed in $H_2$ if and only if $\mathcal{R}(U^{\dagger})$ is closed in $H_1$.
\item $(U^{\ast})^\dagger=(U^\dagger)^\ast$.
\item
The orthogonal projection of $H_2$ onto $\mathcal{R}(U)$ is given by $UU^{\dagger}$.
\item
The orthogonal projection of $H_1$ onto $\mathcal{R}(U^{\dagger})$ is given by $U^{\dagger}U$.
\item$\ker U^{\dagger}=\mathcal{R}^{\bot}(U)$ and $\mathcal{R}(U^{\dagger})=(\ker U)^{\bot}$.
 \end{enumerate}
\end{lemma}
The operator $U:H\rightarrow H$ is called a \textit{unitary operator}  if $U^*=U^{-1}$. In this case, it is obvious that $\Vert U\Vert=1$. The following lemma characterizes all orthonormal bases by unitary operators.
\begin{lemma}\cite{ch}
Let $\{e_j\}_{j\in\Bbb J}$ be an orthonormal basis for $H$. Then, the orthonormal bases for $H$ are precisely the sets $\{Ue_j\}_{j\in\Bbb J}$ where $U:H\rightarrow H$ is unitary.
\end{lemma}
Suppose that $W,V$ are closed subspaces of $H$ with $$\dim W=\mathop{\rm rank}P_V=m,$$ then we can check that there is an unitary operator $U$ on $H$ such that $U\Big\vert_{W}=P_V$. Indeed, if we choose $Ue_j=e'_j$ where $\{e_j\}_{j=1}^{m}$ and $\{e'_j\}_{j=1}^{m}$ are orthonormal bases for $W$ and $V$, respectively, then $U$ is unitary.

 Assume that $T,S$ are operators on $H$. The operator $S$ is called to be \textit{unitarily equivalent} to $T$ if there is a unitary operator $U$ on $H$ such that $S=UTU^*$.
\begin{lemma}(\cite{dag}).\label{dag}
Let $L_1\in\mathcal{B}(H_1,H)$ and $L_2\in\mathcal{B}(H_2,H)$ be  on given Hilbert spaces. Then the following assertions are equivalent:
\begin{enumerate}\item
$\mathcal{R}(L_1)\subseteq\mathcal{R}(L_2)$;\item
$L_1L_1^*\leq\lambda^2 L_2L_2^*$ for some $\lambda>0$;\item
there exists a  mapping $U\in\mathcal{B}(H_1,H_2)$ such that $L_1=L_2 U$.
\end{enumerate}

Moreover, if those conditions are valid, then there exists a unique operator $U$ such that
\begin{enumerate}
\item[(a)] $\Vert U\Vert^2=\inf\{\alpha>0 \ \vert \ L_1L_1^*\leq\alpha L_2L_2^*\};$
\item[(b)] $\mathcal{N}(L_1)=\mathcal{N}(U);$
\item[(c)] $\mathcal{R}(U)\subseteq\overline{\mathcal{R}(L_2^*)}.$
\end{enumerate}
\end{lemma}

\begin{definition}[frame]
Let $\{f_j\}_{j\in\Bbb J}$ be a sequence of members of $H$. We say that $\{f_j\}_{j\in\Bbb J}$ is a
frame  for $H$ if there exist $0<A\leq B<\infty$ such that for
each $f\in H$,
\begin{eqnarray}\label{ord}
A\Vert f\Vert^2\leq\sum_{j\in\Bbb J}\vert\langle f,f_j\rangle\vert^2\leq B\Vert f\Vert^2.
\end{eqnarray}
\end{definition}
The constants $A$ and $B$ are called frame bounds. If the right hand of \eqref{ord} holds, we say that $\{f_j\}_{j\in\Bbb J}$ is a Bessel sequence with bound $B$.
We say that $\{f_j\}_{j\in\Bbb J}$ is a $A$-tight frame  for $H$, if we have
$$\sum_{j\in\Bbb J}\vert\langle f,f_j\rangle\vert^2=A\Vert f\Vert^2,$$
for each $f\in H$. If $A=1$, the set $\{f_j\}_{j\in\Bbb J}$ is called a Parseval frame.

It is easy to check that if $\{f_j\}_{j\in\Bbb J}$ is a frame for $H$ with bounds $A$ and $B$, then by \eqref{ord}, the set $\{Pf_j\}_{j\in\Bbb J}$ is a frame for $PH$ with the same bounds where, $P$ is an orthonormal projection on $H$. We say that two frames $\{f_j\}_{j\in\Bbb J}$ and $\{g_j\}_{j\in\Bbb J}$ are \textit{unitarily equivalent}  if there is a unitary operator $U$ on $H$ such that $g_j=Uf_j$. In this case, we treat two frames as the same.

Let $\{f_j\}_{j\in\Bbb J}$ be a Bessel sequence, then the synthesis and the analysis operators of a given frame are defined by
\begin{align}\label{ops}
T:&\ell^{2}(\Bbb N)\rightarrow H,\qquad T\lbrace c_j\rbrace_{j\in\Bbb J}=\sum_{j\in\Bbb J}c_j f_j,\\
T^*&:H\rightarrow\ell^{2}(\Bbb N) ,\qquad T^*f=\{\langle f,f_j\rangle\}_{j\in\Bbb J}.\label{ops2}
\end{align}
Now, the frame operator is defined by $S=TT^*$ and this is an invertible and positive operator on $H$ and also(see \cite{ch})
$$Sf=\sum_{j\in\Bbb J}\langle f,f_j\rangle f_j,$$
and
$$\langle Sf,f\rangle=\sum_{j\in\Bbb J}\vert\langle f,f_j\rangle\vert^2,$$
for each $f\in H$.
The following result characterizes all Parseval frames as a projection of an orthonormal basis from a larger space.
\begin{lemma}\label{n}\cite{casazza2}
A set $\{f_j\}_{j\in\Bbb J}$ in $H$ is a Parseval frame if and only if there is a larger Hilbert space $M\supseteq H$ and an orthonormal basis $\{e_j\}_{j\in\Bbb J}$ for $M$ so that the orthonormal projection $P$ of $M$ onto $H$ satisfies $f_j=P e_j$ for each $j\in\Bbb J$.
\end{lemma}

In next part, we review notations of $K$-frames  from \cite{arab, gal, xx}.
\begin{definition}[$K$-frame]\cite{gal}
Let $\{f_j\}_{j\in\Bbb J}$ be a sequence of members of $H$ and
$K\in\mathcal{B}(H)$. We say that $\{f_j\}_{j\in\Bbb J}$ is a
$K$-frame  for $H$ if there exist $0<A\leq B<\infty$ such that for
each $f\in H$,
\begin{eqnarray}\label{kframe}
A\Vert K^*f\Vert^2\leq\sum_{j\in\Bbb J}\vert\langle f,f_j\rangle\vert^2\leq B\Vert f\Vert^2.
\end{eqnarray}
\end{definition}
 We say that $\{f_j\}_{j\in\Bbb J}$ is a \textit{equal-norm-$K$-frame}  for $H$, when $\{f_j\}_{j\in\Bbb J}$ is a $K$-frame and $\Vert f_j\Vert=c$ for every $j\in\Bbb J$. When $c=\Vert K \Vert$, we call $\{f_j\}_{j\in\Bbb J}$ a \textit{$K$-norm-frame} .
We say that $\{f_j\}_{j\in\Bbb J}$ is a tight $K$-frame  for $H$, if we have
$$\sum_{j\in\Bbb J}\vert\langle f,f_j\rangle\vert^2=A\Vert K^*f\Vert^2,$$
for each $f\in H$. If $A=1$, the set $\{f_j\}_{j\in\Bbb J}$ is called a Parseval $K$-frame.

Since every $K$-frame is a Bessel sequence, then the synthesis and analysis operators is defined by \eqref{ops} and \eqref{ops2}.
But, the frame operator $S=TT^*$ is not invertible on $H$. However, $S:\mathcal{R}(K)\rightarrow S\big(\mathcal{R}(K)\big)$ is invertible  when the operator $K$ has closed range (\cite{xx}). So, we can construct a Parseval frame for $\mathcal{R}(K)$ by the following result.
\begin{proposition}\label{pro2.1}
Let $\{f_j\}_{j\in\Bbb J}$ be a $K$-frame for $H$ and $K$ has closed range. If $S\big(\mathcal{R}(K)\big)=\mathcal{R}(K)$, then
\begin{enumerate}
\item[(I)] $\{S^{\frac{1}{2}}P_{\mathcal{R}(K)}f_j\}_{j\in\Bbb J}$ is a Parseval frame for $\mathcal{R}(K)$.
\item[(II)] $\{S^{\frac{1}{2}}P_{\mathcal{R}(K)}f_j\}_{j\in\Bbb J}$ is a Parseval $P_{\mathcal{R}(K)}$-frame for $H$.
\end{enumerate}
\end{proposition}
\begin{proof}
(I). Since the operator $S$ is invertible and positive on $\mathcal{R}(K)$, therefore $S^{-1}\Big\vert_{S\big(\mathcal{R}(K)\big)}$ has a unique positive root as $S^{-\frac{1}{2}}$ and this operator commutes with $S^{-1}$ and so with $S$. Now, for each $f\in\mathcal{R}(K)$ we have
\begin{align*}
f&=S^{-1}S\vert_{\mathcal{R}(K)}f\\
&=S^{-\frac{1}{2}}S\vert_{\mathcal{R}(K)}S^{-\frac{1}{2}}f\\
&=\sum_{j\in\Bbb J}\langle S^{-\frac{1}{2}}f, P_{\mathcal{R}(K)}f_j\rangle S^{-\frac{1}{2}}P_{\mathcal{R}(K)}f_j\\
&=\sum_{j\in\Bbb J}\langle f, S^{-\frac{1}{2}}P_{\mathcal{R}(K)}f_j\rangle S^{-\frac{1}{2}}P_{\mathcal{R}(K)}f_j.
\end{align*}
Therefore,
$$\Vert f\Vert^2=\langle f,f\rangle=\sum_{j\in\Bbb J}\vert\langle f,S^{-\frac{1}{2}}P_{\mathcal{R}(K)}f_j\rangle\vert^2.$$
This completes the proof.

(II). Take $f\in H$, so we can write $f=f_1+f_2$ where $f_1\in\mathcal{R}(K)$ and $f_2\in\mathcal{R}(K)^{\perp}$. Then
$$\langle S^{-\frac{1}{2}}P_{\mathcal{R}(K)}f_i, f_2\rangle=0.$$
Thus, we compute that
\begin{align*}
\sum_{j\in\Bbb J}\vert\langle f, S^{-\frac{1}{2}}P_{\mathcal{R}(K)}f_j\rangle\vert^2&=\sum_{j\in\Bbb J}\vert\langle f_1, S^{-\frac{1}{2}}P_{\mathcal{R}(K)}f_j\rangle+\langle f_2, S^{-\frac{1}{2}}P_{\mathcal{R}(K)}f_j\rangle\vert^2\\
&=\sum_{j\in\Bbb J}\vert\langle f_1, S^{-\frac{1}{2}}P_{\mathcal{R}(K)}f_j\rangle\vert^2+
\sum_{j\in\Bbb J}\vert\langle f_2, S^{-\frac{1}{2}}P_{\mathcal{R}(K)}f_j\rangle\vert^2\\
&\qquad +2\sum_{j\in\Bbb J}\mathop{\rm Re}\langle f_1, S^{-\frac{1}{2}}P_{\mathcal{R}(K)}f_j\rangle\langle  S^{-\frac{1}{2}}P_{\mathcal{R}(K)}f_j ,f_2\rangle\\
&=\sum_{j\in\Bbb J}\vert\langle f_1, S^{-\frac{1}{2}}P_{\mathcal{R}(K)}f_j\rangle\vert^2.
\end{align*}
Now, by item (I), we conclude that
$$\sum_{j\in\Bbb J}\vert\langle f, S^{-\frac{1}{2}}P_{\mathcal{R}(K)}f_j\rangle\vert^2=\Vert P_{\mathcal{R}(K)}f\Vert^2.$$
\end{proof}
\begin{remark}
Assume that $F=\{f_j\}_{j\in\Bbb J}$ is a Parseval $K$-frame for $H$ and $K$ has closed range. It is obvious that $\ker KK^*=\ker K^*$, therefore we get $\mathcal{R}(KK^*)=\mathcal{R}(K)$. If $S$ is the frame operator of $F$, then $\langle Sf,f\rangle=\Vert K^*f\Vert^2$. So, $S=KK^*$ and we conclude that
$\mathcal{R}(S)=\mathcal{R}(K)$.
\end{remark}
The following theorem gives a characterization of every $K$-frames using linear bounded operators.
\begin{lemma}\label{lem3}\cite{gal}
Let $\{f_j\}_{j\in\Bbb J}$ be members of $H$. Then $\{f_j\}_{j\in\Bbb J}$ is a $K$-frame if and only if there exists a linear bounded operator $T:\ell^2(\Bbb J)\rightarrow H$ such that $f_j=T\delta_j$ and $\mathcal{R}(K)\subseteq\mathcal{R}(T)$, where $\{\delta_j\}_{j\in\Bbb J}$ is an orthonormal basis for $\ell^2(\Bbb J)$.
\end{lemma}
In the proof of Lemma \ref{lem3}, we can check that the operator $T$ is the synthesis operator (see Theorem 4 in \cite{gal}).
\begin{proposition}\label{lem}
Let $\{f_j\}_{j\in\Bbb J}$ be a Parseval $K$-frame for $H$ and $\{\lambda_j\}_{j=1}^{n}$ be the eigenvalues for the frame operator $S$ where $n=\dim H$. Then $\lambda_j=\Vert K\Vert^2$ for each $j=1,2,\cdots,n$ and
$$\sum_{j\in\Bbb J}\Vert f_j\Vert^2=n\Vert K\Vert^2.$$
\end{proposition}
\begin{proof}
First, we notice that  the equality:
\begin{equation}\label{eig}
\sum_{j=1}^{n}\lambda_j=\sum_{j\in\Bbb J}\Vert f_j\Vert^2
\end{equation}
holds for $K$-frames which has been presented in \cite{ch} for ordinary frames. Suppose that $f\in H$ and $\lambda_j$ is an eigenvalue of $S$ for arbitrary $j\in\Bbb J$. Since $S=KK^*$, therefore, $\Vert K^*f\Vert^2=\lambda_j\Vert f\Vert^2$ and we conclude that $\Vert K^*\Vert^2=\lambda_j$ for each $j=1,2,\cdots,n$. Thus, by \eqref{eig}, this completes the proof.
\end{proof}
\begin{definition}\cite{arab}
Let $F:=\{f_j\}_{j\in\Bbb J}$ be a $K$-frame and $G:=\{g_j\}_{j\in\Bbb J}$ be a Bessel sequence for $H$ with synthesis operators $T$ and $\Theta$ respectively. We say that $G$ is a $K$-dual for $F$ if
\begin{eqnarray}\label{dual}
T\Theta^*=K.
\end{eqnarray}
\end{definition}
In this case, from \eqref{dual} we can write for each $f\in H$,
\begin{equation}\label{dual2}
Kf=\sum_{j\in\Bbb J}\langle f,g_j\rangle f_j,
\end{equation}
and it is easy to see that $G$ is a $K^*$-frame for $H$.
In the following, we can study to characterize all $K$-dual of a $K$-frame.
\begin{lemma}\cite{arab}\label{arab}
Let $F:=\{f_j\}_{j\in\Bbb J}$ be a $K$-frame. Then $\{g_j\}_{j\in\Bbb J}$ be a $K$-dual of $F$ if and only if $\{g_j\}_{j\in\Bbb J}=\{V\delta_j\}_{j\in\Bbb J}$ where $\{\delta_j\}_{j\in\Bbb J}$ is the standard orthonormal basis of $\ell^2$ and $V:\ell^2\rightarrow H$ is a bounded operator such that $TV^*=K$. In this case, $\{g_j\}_{j\in\Bbb J}$ is in fact a Parseval $V$-frame.
\end{lemma}
\section{main results}
The following which is a general case of Theorem 3.2 in \cite{casazza2}, shows that each finite set of $K$-norm vectors can be extended to become a $K$-norm of a $K$-frame.
\begin{theorem}\label{t1}
If $F:=\{f_j\}_{j=1}^{M}$ is a set of $K$-norm vectors in $H$, then there is a $K$-norm-frame for $H$ which contains the set $F$.
\end{theorem}
\begin{proof}
Let $1\leq j\leq M$ and $\{e_{ij}\}_{i\in\Bbb I}$ be an orthonormal basis for $H$ which contains the vector $u_j$ where $u_j:=\dfrac{f_j}{\Vert K\Vert}$. For any $f\in H$ we have
$$M\Vert K^*f\Vert^2\leq M\Vert K\Vert^2\Vert f\Vert^2=\sum_{j=1}^{M}\sum_{i\in\Bbb I}\vert\langle f,\Vert K\Vert e_{ij}\rangle\vert^2.$$
Therefore, $\{\Vert K\Vert e_{ij}\}_{j=1, i\in\Bbb I}^{M}$ is a $K$-frame for $H$ with bounds $M$ and $M\Vert K\Vert^2$ and is made up $K$-norm vectors.
\end{proof}
\begin{corollary}
If $K$ has closed range, $F:=\{f_j\}_{j=1}^{M}$ is a set of $K$-norm vectors in $H$ and $\Vert K\Vert\Vert K^{\dagger}\Vert=1$, then there is a $K$-norm-tight frame for $\mathcal{R}(K)$ which contains the set $F$.
\end{corollary}
\begin{proof}
Chose $f\in\mathcal{R}(K)$, so
$$\Vert K\Vert\Vert f\Vert=\Vert K\Vert\Vert(K^{\dagger})^*K^*f\Vert\leq\Vert K^*f\Vert,$$
hence we get $\Vert K^*f\Vert=\Vert K\Vert\Vert f\Vert$. Now,  by Theorem \ref{t1}, if $\{e_{ij}\}_{i\in\Bbb I}$ is an orthonormal basis for $\mathcal{R}(K)$, then $\{\Vert K\Vert e_{ij}\}_{j=1, i\in\Bbb I}^{M}$ is a $K$-norm-tight frame for $\mathcal{R}(K)$.
\end{proof}
The next result is the same of Lemma \ref{n} for Parseval $K$-frames.
\begin{theorem}\label{tt1}
Let $\{f_j\}_{j\in\Bbb J}$ be a Parseval $K$-frame for $H$ and $K$ be closed range. Then there is a larger Hilbert space $M\supseteq\mathcal{R}(K^{\dagger})$ and an orthonormal basis $\{e_j\}_{j\in\Bbb J}$ for $M$ so that the orthonormal projection $P$ of $M$ onto $\mathcal{R}(K^{\dagger})$ satisfies $f_j=KPe_j$ for each $j\in\Bbb J$.
\end{theorem}
\begin{proof}
Assume that  $\{f_j\}_{j\in\Bbb J}$ is a Parseval $K$-frame for $H$. Via Lemma \ref{lem3}, there exists an orthonormal basis $\{e'_j\}_{j\in\Bbb J}$ which $f_j=Te'_j$ for all $j\in\Bbb J$. Since $\{f_j\}_{j\in\Bbb J}$ is a Parseval $K$-frame, we have $KK^*=TT^*$ and by Lemma \ref{dag} we can get $\mathcal{R}(T)=\mathcal{R}(K)$. Therefore $f_j\in\mathcal{R}(K)$. Suppose that $f\in\mathcal{R}(K^{\dagger})$, by Lemma \ref{lem2} we have
$$\sum_{j\in\Bbb J}\vert\langle f,K^{\dagger}f_j\rangle\vert^2=\sum_{j\in\Bbb J}\vert\langle (K^{\dagger})^* f,f_j\rangle\vert^2=\Vert K^*(K^{\dagger})^*f\Vert^2=\Vert f\Vert^2.$$
Hence, $\{K^{\dagger}f_j\}_{j\in\Bbb J}$ is a Parseval frame for $\mathcal{R}(K^{\dagger})$. So, by Lemma \ref{n}, there exists a larger Hilbert space $M\supseteq\mathcal{R}(K^{\dagger})$ and an orthonormal basis $\{e_j\}_{j\in\Bbb J}$ for $M$ so that $K^{\dagger}f_j=Pe_j$ for each $j\in\Bbb J$ or $f_j=KPe_j$.
\end{proof}
Next result is a general case of Theorem 3.4 in \cite{casazza2} for $K$-frames. In the following, we let $\dim H=n$ with an orthonormal basis $\{e_j\}_{j=1}^{n}$.
\begin{theorem}
Let $K$ be closed range, $P$ be a  orthonormal projection on $H$ onto $\mathcal{R}(K^{\dagger})$ such that $KP$ be a rank-$m$. Define
\begin{align*}
\mathscr{K}&=\{\mbox{all Parseval K-frames for}\, KPH\}\\
\mathscr{M}&=\{W\prec H,\quad \dim W=m\}.
\end{align*}
Then there exists a natural one-to-one correspondence between $\mathscr{K}$ and $\mathscr{M}$.
\end{theorem}
\begin{proof}
For each $F:=\{f_j\}_{j=1}^{n}\in\mathscr{K}$, by Theorem \ref{t1}, there is an orthonormal basis $\{e'_j\}_{j=1}^{n}$ such that $f_j=KPe'_j$ for any $1\leq j\leq n$. Define a unitary operator $U_F$ on $H$ by $U_F e_j=e'_j$. So we have $f_j=KPU_F e_j$ which is unitarily equivalent to $f'_j:=U_F^*KPUe_j$.
Define
\begin{align*}
\Phi:&\mathscr{K}\longrightarrow \mathscr{M}\\
\Phi\{f_j\}_{j=1}^{n}&=U_F^*KPU_F H.
\end{align*}
It is clear that the operator $\Phi$ is well-defined. Assume that $W\in\mathscr{M}$, thus there exists a unitary operator $U$ on $H$ such that $UW=KPH$. So, we get $U^*KPUH=W$ while $\{KPe_j\}_{j=1}^{n}\in\mathscr{K}$ which corresponds to $W$. This means that $\Phi$  is surjective. Finally, suppose that $G:=\{g_j\}_{j=1}^{n}\in\mathscr{K}$ and $V_G$ is a unitary operator on $H$ such that $U^*_FKPU_F H=V_G^*KPV_G H$ which
$$V_G e_j=e''_j\, , \quad g_j=KPe''_j=KPV_G e_j.$$
Since $U^*_FKPU_F e_j=V_G^*KPV_G e_j$ for each $1\leq j\leq n$, we have
$U_F^* f_j=V^*_G g_j$. Hence, $f_j$ and $g_j$ are unitarily equivalent for each $j$. So, $F$ and $G$ are two same $K$-Parseval frames, then $\Phi$ is injective.
\end{proof}
\begin{theorem}
Let $\{f_j\}_{j\in\Bbb J}$ be a Parseval $K$-frame for $H$ with the synthesis operator $T$. If $\{g_j\}_{j\in\Bbb J}$ is a  $K$-dual for $\{f_j\}_{j\in\Bbb J}$ with the synthesis operator $\Theta$ which $\{g_j\}_{j\in\Bbb J}$ is a Parseval $K^*$-frame for $H$, then, for each $f\in H$ we have
$$\Vert (T^*-\Theta^*K^*)f\Vert^2=\Vert KK^*f\Vert^2-\Vert K^*f\Vert^2.$$
Moreover, if $K$ has closed range, we can get the following inequality:
$$\Vert (KK^*)^{-1}\Vert^{-2}(1-\Vert K\Vert^2)\leq \Vert T-K\Theta\Vert^2\leq \Vert K\Vert^4.$$
\end{theorem}
\begin{proof}
Since
$T\Theta^*=K$ and $TT^*=KK^*$, so we get $T(T^*-\Theta^*K^*)=0$. Now, for any $f,g\in H$ we have
$$\langle T^*g,(T^*-\Theta^*K^*)f\rangle=\langle g,T(T^*-\Theta^*K^*)f\rangle=0.$$
Hence, we compute that
\begin{align*}
\Vert \Theta^*K^*f\Vert^2&=\Vert T^*f-T^*f+\Theta^*K^*f\Vert^2\\
&=\Vert T^*f-(T^*-\Theta^*K^*)f\Vert^2\\
&=\Vert T^*f\Vert^2+\Vert (T^*-\Theta^*K^*)f\Vert^2\\
&\qquad-\langle T^*f, (T^*-\Theta^*K^*)f\rangle-\overline{\langle T^*f, (T^*-\Theta^*K^*)f\rangle}\\
&=\Vert T^*f\Vert^2+\Vert (T^*-\Theta^*K^*)f\Vert^2.
\end{align*}
But, via the hypothesis, we have $\Vert T^*f\Vert^2=\Vert K^*f\Vert^2$ and $\Vert \Theta^*f\Vert^2=\Vert Kf\Vert^2$. For the second part, since $K$ has closed range,  it is easy to check that the operator $KK^*:\mathcal{R}(K)\rightarrow\mathcal{R}(K)$ is invertible. So, for any $0\neq f\in\mathcal{R}(K)$ we have $\Vert f\Vert^2\leq\Vert (KK^*)^{-1}\Vert^2\Vert KK^*f\Vert^2$. Hence,
$$\Vert T-K\Theta\Vert^2=\Vert T^*-\Theta^*K^*\Vert^2\geq\frac{\Vert (T^*-\Theta^*K^*)f\Vert^2}{\Vert f\Vert^2}\geq\Vert (KK^*)^{-1}\Vert^{-2}(1-\Vert K\Vert^2).$$
This completes the proof.
\end{proof}
In next result which is a general case of Proposition 3.3 in \cite{casazza2} , we define $K^{\natural}$ as  a left inverse of $K$.
\begin{theorem}
Let $F=\{f_j\}_{j=1}^{m}$ be a Parseval $K$-frame for the finite Hilbert space $H$ with the synthesis operator $T$ such that $T^*H\perp U^*H$ and $K^{\natural}f_j\perp Uf_j$ for each $1\leq j\leq m$ where $U:H\rightarrow\ell^{2}$ and $K^{\natural}\vert_F$ are an isometry. Then $\{f_j\}_{j=1}^{m}$ has infinitely many equal-norm dual $K$-frames.
\end{theorem}
\begin{proof}
Choose $V=aU$ where $a\neq0$. Since $(K^{\natural}T+U)T^*=K^*$, so via Lemma \ref{arab}, if $g_j=(K^{\natural}T+V)\delta_j$ for $1\leq j\leq m$, then $\{g_j\}_{j=1}^{m}$ is a dual $K$-frame of $\{f_j\}_{j=1}^{m}$ where $\{\delta_j\}_{j=1}^{m}$ is the orthonormal basis for $\ell^2$. Assume that $P:\ell^2\rightarrow T^*H$ is an orthonormal projection so that $P\delta_j=T^*f_j$ for any $1\leq j\leq m$. Since $T\delta_j=f_j$, therefore for each $1\leq j\leq m$,
$$g_j=K^{\natural}f_j+V(I-P)\delta_j.$$
But, it is easy to check that $\langle K^{\natural}f_j,V(I-P)\delta_j\rangle=0$ and also via Proposition \ref{lem}, if we suppose $n=\dim H$, then we can get
$$\Vert f_j\Vert^2=\frac{n}{m}\Vert K\Vert^2,$$
and
$$\Vert P\delta_j\Vert^2=\Vert T^*f_j\Vert^2=\Vert K^*f_j\Vert^2=\frac{n}{m}\Vert K\Vert^4.$$
Thus, for each $1\leq j\leq m$, we conclude that
\begin{align*}
\Vert g_j\Vert^2&=\langle K^{\natural}f_j+V(I-P)\delta_j,K^{\natural}f_j+V(I-P)\delta_j\rangle\\
&=\Vert K^{\natural}f_j\Vert^2+\Vert V(I-P)\delta_j\Vert^2\\
&=\Vert f_j\Vert^2+a^2\Vert (I-P)\delta_j\Vert^2\\
&=\Vert f_j\Vert^2+a^2\big(1+\Vert K^*f_j\Vert^2-2\Vert f_j\Vert^2\big)\\
&=a^2+\Big(1-2a^2\Big)\frac{n}{m}\Vert K\Vert^2+\frac{a^2n}{m}\Vert K\Vert^4.
\end{align*}
\end{proof}

\end{document}